\documentclass[oneside,reqno]{amsart}
\usepackage{amsfonts}

\usepackage{amsmath}
\usepackage{enumerate}

\newtheorem{theorem}{Theorem}[section]

\theoremstyle{definition}
\newtheorem{definition}[theorem]{Definition}

\numberwithin{equation}{section}

\newcommand{\blankbox}[5]

\input{tcilatex}

\begin{document}
\title{Weak and Strong Boundedness for $p$-Adic Fractional Hausdorff
Operator and Its Commutator}
\author{Naqash Sarfraz$^{1}$}
\author{Fer\'{\i}t G\"{u}rb\"{u}z $^{2,*}$}

\begin{abstract}
In this paper, boundedness of Hausdorff operator on weak central Morrey
space is obtained. Furthermore, we investigate the weak bounds of $p$-adic
fractional Hausdorff Operator on weighted $p$-adic weak Lebesgue Space. We
also obtain the sufficient condition of commutators of $p$-adic fractional
Hausdorff Operator by taking symbol function from Lipschitz space. Moreover,
strong type estimates for fractional Hausdorff Operator and its commutator
on weighted $p$-adic Lorentz space are also acquired.
\end{abstract}

\maketitle

\footnote{$^{}$Department of Mathematics, Quaid-I-Azam University 45320,
Islamabad 44000, Pakistan\newline
$^{2}$Hakkari University, Faculty of Education, Department of Mathematics
Education, Hakkari 30000, Turkey\newline
$^*$Corresponding Author: feritgurbuz@hakkari.edu.tr(F. G\"{u}rb\"{u}z);
naqashawesome@gmail.com (N.Sarfraz)}

\section{\textbf{Introduction}}

Any non zero rational number $x$ can be written as $x=p^\gamma m/n,$ where $p
$ is fixed prime and $m,n$ are integers. It is mandatory that all these
numbers are coprime to each other. The $p$-adic absolute value of $x$ is as
follows: 
\begin{equation*}
\{|x|_{p}:x\in\mathbb{Q}_p\}=\{p^{-\gamma}:\gamma\in\mathbb{Z}\}\cup\{0\}.
\end{equation*}
In the above set, $\mathbb{Q}_p$ is the field of $p$-adic numbers and is the
completion of the field of rational number $\mathbb{Q}$ with respect to
non-Archimedean $p$-adic norm $|\cdot|_p.$ The $p$-adic absolute value $%
|\cdot|_p$ satisfies all the conditions of real norm together with so called
strong triangular inequality, 
\begin{equation}  \label{NE}
|x+y|_p\le\max\{|x|_p,|y|_p\}.
\end{equation}
Furthermore, if $|x|_{p}\neq|y|_{p},$ then (\ref{NE}) takes the form: 
\begin{equation}
|x+y|_p=\max\{|x|_p,|y|_p\}.
\end{equation}%
A $p$-adic number $x\in \mathbb{Q}_p\backslash\{0\}$ can also be represented
in canonical form (see \cite{VVZ}) as:%
\begin{equation}  \label{E3}
x=p^\gamma\sum_{k=0}^{\infty}\alpha_kp^k,
\end{equation}
where $\alpha_k,\gamma\in\mathbb{Z},\alpha_0\ne 0, \alpha_k\in\frac{\mathbb{Z%
}}{p \mathbb{Z}_{p}}.$ The above series converges in $p$-adic norm because
of the fact $|p^\gamma\alpha_{k}p^{k}|_{p}= p^{-\gamma-k}.$

The space $\mathbb{Q}_p^n=\mathbb{Q}_p\times...\times\mathbb{Q}_p$ consists
of points $\mathbf{x}=(x_{1},x_{2},...,x_{n}),$ where $x_k\in\mathbb{Q}%
_p,k=1,2,...,n.$ The $p$-adic norm can also be defined on higher dimensional
space $\mathbb{Q}_p^n$ as  
\begin{equation}  \label{ll}
|\mathbf{x}|_{p}=\max_{1\leq k \leq n}|x_{k}|_{p}.
\end{equation}
The norm in (\ref{ll}) is non-Archimedean one.

Let us represent 
\begin{equation*}
B_{\gamma}(\mathbf{a})=\{\mathbf{x} \in \mathbb{Q}_p^n:|\mathbf{x}-\mathbf{a}%
|_{p}\leq p^{\gamma}\},
\end{equation*}
the ball with center at $\mathbf{a} \in \mathbb{Q}_p^n$ and radius $%
p^{\gamma}$. In a same way, represent by 
\begin{equation*}
S_{\gamma}(\mathbf{a})=\{\mathbf{x} \in \mathbb{Q}_p^n:|\mathbf{x}-\mathbf{a}%
|_{p}=p^{\gamma}\},
\end{equation*}
the sphere with center at $\mathbf{a} \in \mathbb{Q}_p$ and radius $%
p^{\gamma}$. When $\mathbf{a}=\mathbf{0},$ we just represent $B_\gamma(%
\mathbf{0})=B_\gamma$ and $S_\gamma(\mathbf{0})=S_\gamma.$ Additionally, for
each $\mathbf{a}_0\in\mathbb{Q}_p^n, \ \mathbf{a}_0+B_\gamma=B_\gamma(%
\mathbf{a}_0)$ and $\mathbf{a}_0+S_\gamma=S_\gamma(\mathbf{a}_0)$.

The locally compact commutative group under addition of $\mathbb{Q}_p^n$
makes sure the existence of additive positive Haar measure $d\mathbf{x}$
invariants under shift $d(x+a)=dx, a\in\mathbb{Q}_p.$ The measure $dx$ is
unique if the following equality normalize it 
\begin{equation*}
\int_{B_{0}(\mathbf{0})}d\mathbf{x}=|B_{0}(\mathbf{0})|=1,
\end{equation*}
where $|B|$ denotes the Haar measure of a subset $B$ of $\mathbb{Q}_p^n,$
and is measurable. Also, an easy calculation shows $|B_{\gamma}(\mathbf{a}%
)|=p^{n\gamma}$, $|S_{\gamma}(\mathbf{a})|=p^{n\gamma}(1-p^{-n})$, for any $%
\mathbf{a}\in \mathbb{Q}_p^n$.

$p$-adic analysis is a key tool to describe Kohlrausch-Williams-Watts law,
the power decay law and the logarithmic decay law (see \cite{ABO}). $p$-adic
analysis is natural to non-Archimedean spaces (see \cite{{ABKO1},{ABKO}}).
Its applications in theoretical physics and theoretical biology can be found
in \cite{{ADFV},{DGSK},{PS},{VV}}. $p$-adic analysis has also cemented its
role in $p$-adic pseudo-differential equations and stochastic process, see
for example \cite{{K},{VVZ}}. Besides, in this day and age many researchers
have shown heaps of interest in the study of wavelet and harmonic analysis ,
for instance, (see \cite{{SH},{SR},{SVK}}).

The classical one dimensional Hausdorff operator is defined as: 
\begin{equation*}
h_{\Phi}f(\mathbf{x})=\int_{0}^{\infty}\frac{\Phi(t)}{t}f(\frac{x}{t}%
)dt,\quad x\in\mathbb{R},
\end{equation*}%
\newline
where $\Phi$ is integrable function on $\mathbb{R}^{+}=(0,\infty).$ Anderson
in \cite{A} defined the $n$- dimensional Hausdorff operator\newline
\begin{equation*}
H_{\Phi}f(\mathbf{x})=\int_{\mathbb{R}^{n}}\frac{\Phi(\mathbf{x}/|\mathbf{y}%
|)}{|\mathbf{y}|^{n}}f(\mathbf{y})d\mathbf{y},
\end{equation*}%
\newline
where $\Phi$ is function defined on ${\mathbb{R}}^{n}.$

A.K. Lerner and E. Liflyand\cite{LL} studied the most general matrix
Hausudorff operator which is given by: 
\begin{equation*}
(H_{\Phi,A}f)(\mathbf{x})=\int_{\mathbb{R}^{n}}\Phi(t)f(A(t)\mathbf{x})dt,
\end{equation*}
where $A(t)$ is the $n\times n$ invertible matrix almost everywhere in the
support of $\Phi.$

An extension of Hausdorff operator is the fractional Hausdorff operator
which was studied by Lin and Shan \cite{LS} is as follows: \newline
\begin{equation*}
H_{\Phi,\beta}(f)(\mathbf{x})=\int_{\mathbb{R}^{n}}\frac{\Phi(|\mathbf{x}|/|%
\mathbf{y}|)}{|\mathbf{y}|^{n-\beta}}f(\mathbf{y})d\mathbf{y},\quad
0\leq\beta<n.
\end{equation*}

Later, the weak and strong estimates of two kinds of multilinear fractional
Hausdorff operator on Lebesgue space were studied by Fan and Zhao, see \cite%
{FZ}. Gao and Zhao \cite{GZ} obtained the sharp weak bounds for Hausdorff
operators. For more details about weak bounds we refer some last
publications including \cite{GHZ,HJ}.

Inspiring from above results we define the $p$-adic fractional Hausdorff
operator 
\begin{equation*}
H_{\Phi,\beta}(f)(x)=\int_{\mathbb{Q}_p^n}\frac{\Phi(\mathbf{x}|\mathbf{y}%
|_{p})}{|\mathbf{y}|_{p}^{n-\beta}}f(\mathbf{y})d\mathbf{y},\quad
0\leq\beta<n.
\end{equation*}%
\newline
Here, we consider $|y|_{p}$ is equal to some power of $p\in\mathbb{Q}_p^n.$%
\newline
If $\beta=0,$ we get the Hausdorff operator which is defined by 
\begin{equation*}
H_{\Phi}(f)(x)=\int_{\mathbb{Q}_p^n}\frac{\Phi(\mathbf{x}|\mathbf{y}|_{p})}{|%
\mathbf{y}|_{p}^{n}}f(\mathbf{y})d\mathbf{y}
\end{equation*}
We also define the commutators $H^{b}_{\Phi,\beta}(f)=bH_{\Phi,\beta}(f)-H_{%
\Phi,\beta}(bf),$ where $b\in\Lambda_{\delta}(\mathbb{Q}_p^n),$ for some $%
0<\delta\leq1.$ The aim of this paper is to study the boundedness of
Hausdorff operator $H_{\Phi}$ on $p$-adic weak central Morrey space.
Moreover, we get the weak bounds of $H_{\Phi,\beta}$ and $H_{\Phi,\beta}^{b}$
from $L^{r,\infty}(|\mathbf{x}|_{p}^{\gamma},\mathbb{Q}_{p}^{n})$ to $L^{q}(|%
\mathbf{x}|_{p}^{\alpha},\mathbb{Q}_{p}^{n}).$ It is worthwhile to mention
here that the symbol function $b\in\Lambda_{\delta}(\mathbb{Q}_p^n).$ In
addition, strong type estimates of weighted $p$-adic Lorentz space in both
cases are also acquired. Throughout this article the letter $C$ denotes a
constant independent from essential values. Also the $A\preceq B$ denotes
that there exists a constant $C$ such that $A\leq CB.$

\section{Preliminaries}

Let $w(\mathbf{x})$ be a weight function on $\mathbb{Q}_p^n$ which is
nonnegative and locally integrable function. Let $L^q(w,\mathbb{Q}%
_p^n),(0<q<\infty)$ be the space of all complex-valued functions $f$ on $%
\mathbb{Q}_p^n$ such that

\begin{equation*}
\|f\|_{L^{q}(w,\mathbb{Q}_{p}^{n})}=\bigg(\int_{\mathbb{Q}_{p}^{n}}|f(%
\mathbf{x})|^{p}w(\mathbf{x})d\mathbf{x}\bigg)^{1/q}.
\end{equation*}

Now, we define the weighted $p$-adic weak Lebesgue space, as a measurable
function $f$ belongs to $L^{q,\infty}(w,\mathbb{Q}_{p}^{n})$ if\newline
\begin{equation*}
\|f\|_{L^{q,\infty}(w,\mathbb{Q}_{p}^{n})}=\sup_{\lambda>0}\lambda w\bigg(\{%
\mathbf{x}\in\mathbb{Q}_{p}^{n}:|f(\mathbf{x})|>\lambda\}\bigg)^{1/q}<\infty,
\end{equation*}%
\newline
where\newline
\begin{equation*}
w(\{\mathbf{x}\in\mathbb{Q}_{p}^{n}:|f(\mathbf{x})|>\lambda\})=\int_{\{%
\mathbf{x}\in\mathbb{Q}_{p}^{n}:|f(\mathbf{x})|>\lambda\}}w(\mathbf{x})d%
\mathbf{x}.
\end{equation*}%
\newline
Further, when $b\in\Lambda_{\delta}(\mathbb{Q}_p^n),$ for $0<\delta<1,$ the
homogeneous Lipschitz space is defined as follows: 
\begin{equation*}
\|b\|_{\Lambda_{\delta}(\mathbb{Q}_p^n)}= \sup_{\mathbf{x}, \mathbf{h}\in%
\mathbb{Q}_p^n, \mathbf{h}\neq 0}\frac{|b(\mathbf{x}+\mathbf{h})-b(\mathbf{x}%
)|}{|\mathbf{h}|_{p}^{\delta}}<\infty.
\end{equation*}
The distribution function of $f\in\mathbb{Q}_{p}^{n}$ with a measure $w(%
\mathbf{x})d\mathbf{x}$ is defined as: 
\begin{eqnarray*}
\begin{aligned}\mu_{f}(\lambda)=w\{\mathbf{x}\in\mathbb{Q}_{p}^{n}:|f(%
\mathbf{x})|>\lambda\} \end{aligned}
\end{eqnarray*}
The decreasing rearrangement of $f$ with respect to measure $w(\mathbf{x})d%
\mathbf{x}$ is as follows 
\begin{eqnarray*}
\begin{aligned}f^{*}(t)=\inf\{\lambda>0:\mu_{f}(\lambda)\leq t\},
t\in\mathbb{R}_{+}. \end{aligned}
\end{eqnarray*}
Next, we define the weighted $p$-adic Lorentz Space $L^{q,s}(w,\mathbb{Q}%
_{p}^{n})$ which is the collection of all functions $f$ such that $%
\|f\|_{L^{q,s}(w,\mathbb{Q}_{p}^{n})}<\infty,$ where 
\begin{eqnarray*}
\begin{aligned}\|f\|_{L^{q,s}(w,\mathbb{Q}_{p}^{n})}&={\begin{cases}\bigg(%
\frac{s}{q}\int_{0}^{\infty}[t^{1/q}f^{*}(t)]^{s}\frac{dt}{t}\bigg) , &
\text{if } 1\leq s<\infty,\\ \sup_{t>0}t^{1/q}f^{*}(t), & \text{if }
s=\infty.\end{cases}}\\ \end{aligned}
\end{eqnarray*}
It is easy to see that 
\begin{equation}  \label{pa}
\|f\|_{L^{r,s}(w,\mathbb{Q}_{p}^{n})}\leq C\|f\|_{L^{q,s}(w,\mathbb{Q}%
_{p}^{n})}, f\in L^{q,s}(w,\mathbb{Q}_{p}^{n}), 0<q<\infty, 0<s\leq r<\infty.
\end{equation}
If an operator $T$ is bounded from $L^{q,1}(w,\mathbb{Q}_{p}^{n})$ into $%
L^{r,\infty}(w,\mathbb{Q}_{p}^{n}),$ then $T$ is of weak type $(q,r).$ From (%
\ref{pa}), it is obvious to see that 
\begin{equation*}
w\{\mathbf{x}\in\mathbb{Q}_{p}^{n}:|Tf(\mathbf{x})|>\lambda\}\leq
C\lambda^{-r}\|f\|^{r}_{L^{q}(w,\mathbb{Q}_{p}^{n})}, 1\leq q\leq r<\infty,
\end{equation*}
justifies weak type $(q,r)$ for $T$. We will also use the following
Marcinkiewicz Theorem for $w(\mathbf{x})=|\mathbf{x}|_{p}^{\alpha},
\alpha>-n.$

\begin{theorem}
\label{T^} Suppose $1\leq q^{\prime }<q_{0},$ $1\leq r^{\prime },r_{0},
r^{\prime }\neq r_{0},$ $\vartheta\in(0,1)$ and 
\begin{equation*}
\frac{1}{q}=(1-\vartheta)/q^{\prime }+\vartheta/q_{0}, \frac{1}{r}%
=(1-\vartheta)/r^{\prime }+\vartheta/r_{0} .
\end{equation*}%
If $T$ is of weak type $(q^{\prime },r^{\prime })$ and $(q_{0},r_{0}),$ then 
$T$ is bounded from $L^{q,s}(|\mathbf{x}|_{p}^{\alpha},\mathbb{Q}_{p}^{n})$
into $L^{r,s}(|\mathbf{x}|_{p}^{\alpha},\mathbb{Q}_{p}^{n}),$ for all $1\leq
s<\infty.$
\end{theorem}

\begin{definition}
\cite{WF} Let $1\leq q<\infty$ and $-1/q\leq\lambda<0.$ A function $f\in
L^{p}_{loc}(\mathbb{Q}_p^n)$ if 
\begin{eqnarray*}
\begin{aligned}\dot{B}^{q,\lambda}(\mathbb{Q}_p^n)=\sup_{\gamma\in%
\mathbb{Z}}\bigg(\frac{1}{|B_{\gamma}|_{H}^{{1+\lambda
q}}}\int_{B_{\gamma}}|f(\mathbf{x})|^{q}d\mathbf{x}\bigg)^{1/q}.
\end{aligned}
\end{eqnarray*}
When $\lambda=-1/q,$ then $\dot{B}^{q,\lambda}(\mathbb{Q}_p^n)=L^{q}(\mathbb{%
Q}_p^n).$ It is not hard to see that $\dot{B}^{q,\lambda}(\mathbb{Q}_p^n)$
is reduced to \{0\} whenever $\lambda<-1/q.$
\end{definition}

\begin{definition}
\cite{WF1} Let $1\leq q<\infty$ and $-1/q\leq\lambda<0.$ The $p$-adic weak
central Morrey space $W\dot{B}^{q,\lambda}(\mathbb{Q}_p^n)$ is defined as 
\begin{eqnarray*}
\begin{aligned}W\dot{B}^{q,\lambda}(\mathbb{Q}_p^n)=\{f:\|f\|_{W\dot{B}^{q,%
\lambda}(\mathbb{Q}_p^n)}<\infty\}, \end{aligned}
\end{eqnarray*}
where 
\begin{eqnarray*}
\begin{aligned}\|f\|_{W\dot{B}^{q,\lambda}(\mathbb{Q}_p^n)}=\sup_{\gamma\in%
\mathbb{Z}}|B_{\gamma}|_{H}^{-\lambda-1/q}\|f\|_{WL^{q}(B_{\gamma})},%
\end{aligned}
\end{eqnarray*}%
and $\|f\|_{WL^{q}(B_{\gamma})}$ is the local $p$-adic $L^{q}$-norm of $f(x)$
restricted to the ball $B_{\gamma}$, that is 
\begin{eqnarray*}
\begin{aligned}\|f\|_{WL^{q}(B_{\gamma})}=\sup_{\lambda>0}|\{\mathbf{x}\in
B_{\gamma}:|f(\mathbf{x})|>\lambda\}|^{1/q}. \end{aligned}
\end{eqnarray*}
It is clear that if $\lambda=-1/q,$ then $W\dot{B}^{q,\lambda}(\mathbb{Q}%
_p^n)=L^{q,\infty}(\mathbb{Q}_p^n)$ is a $p$-adic weak $L^{q}$ space. Also, $%
\dot{B}^{q,\lambda}(\mathbb{Q}_p^n)\subseteq W\dot{B}^{q,\lambda}(\mathbb{Q}%
_p^n)$ for $1\leq q<\infty$ and $-1/q<\lambda<0.$
\end{definition}

In the upcoming section, we prove the boundedness of Hausdorff operator on $p
$-adic weak central Morrey space.

\section{Boundedness of Hausdorff operator on weak Central Morrey space}

\begin{theorem}
\label{T11}Let $1\leq q<\infty$ and let $-1/q\leq\lambda<0.$ If $\Phi$ is a
radial function, that is $\Phi(\mathbf{x})=\psi(|\mathbf{x}|_{p}),$ where $%
\psi$ is defined in all $p^{k},k\in\mathbb{Z}$ and $f\in\dot{B}^{q,\lambda}(%
\mathbb{Q}_p^n),$ then 
\begin{eqnarray*}
\begin{aligned}\|H_{\Phi}f\|_{W\dot{B}^{q,\lambda}(\mathbb{Q}_p^n)}\leq
K_{1}(1-p^{-n})^{1/q'}|\|f\|_{\dot{B}^{q,\lambda}(\mathbb{Q}_p^n)},
\end{aligned}
\end{eqnarray*}
where $K_{1}=C\int_{0}^{\infty}\psi(t)t^{-n\lambda-1}dt.$
\end{theorem}

\begin{proof}
We first consider%
\begin{eqnarray*}
\begin{aligned}H_{\Phi}=&\int_{\mathbb{Q}_p^n}\frac{\Phi(\mathbf{x}|%
\mathbf{y}|_{p})}{|\mathbf{y}|_{p}^{n}}f(\mathbf{y})d\mathbf{y}\\
=&\sum_{k\in\mathbb{Z}}\int_{S_{k}}\frac{\Phi(\mathbf{x}|\mathbf{y}|_{p})}{|%
\mathbf{y}|_{p}^{n}}f(\mathbf{y})d\mathbf{y}. \end{aligned}
\end{eqnarray*}
By H\"{o}lder's inequality, we have: 
\begin{eqnarray}
\begin{aligned}[b]\label{e1}|H_{\Phi}|\leq&\sum_{k\in\mathbb{Z}}\bigg(\bigg(%
\int_{S_{k}}\frac{|\Phi(\mathbf{x}|\mathbf{y}|_{p})|^{q'}}{|%
\mathbf{y}|_{p}^{nq'}}d\mathbf{y}\bigg)^{1/q'}\bigg(\int_{S_{k}}|f(%
\mathbf{y})|^{q}d\mathbf{y}\bigg)^{1/q}\bigg)\\
\leq&\sum_{k\in\mathbb{Z}}\bigg(\bigg(\int_{S_{k}}\frac{|\Phi(\mathbf{x}|%
\mathbf{y}|_{p})|^{q'}}{|\mathbf{y}|_{p}^{nq'}}d\mathbf{y}\bigg)^{1/q'}%
\bigg(\int_{B_{k}}|f(\mathbf{y})|^{q}d\mathbf{y}\bigg)^{1/q}\bigg)\\
\leq&\|f\|_{\dot{B}^{q,\lambda}(\mathbb{Q}_p^n)}\sum_{k\in%
\mathbb{Z}}|B_{k}|_{H}^{1/q+\lambda}\bigg(\int_{S_{k}}\frac{|\Phi(%
\mathbf{x}|\mathbf{y}|_{p})|^{q'}}{|\mathbf{y}|_{p}^{nq'}}d\mathbf{y}%
\bigg)^{1/q'}. \end{aligned}
\end{eqnarray}
If $|\mathbf{x}|_{p}=p^{l}, l\in\mathbb{Z},$ then $\Phi(\mathbf{x}|\mathbf{y}%
|_{p})=\psi(p^{l-k}),$ then we take: 
\begin{eqnarray*}
\begin{aligned}\sum_{k\in\mathbb{Z}}|B_{k}|_{H}^{1/q+\lambda}\bigg(%
\int_{S_{k}}\frac{|\Phi(\mathbf{x}|\mathbf{y}|_{p})|^{q'}}{|%
\mathbf{y}|_{p}^{nq'}}d\mathbf{y}\bigg)^{1/q'}=&\sum_{k\in%
\mathbb{Z}}p^{kn(1/q+\lambda)}\bigg(\int_{S_{k}}\frac{|%
\psi(p^{l-k})|^{q'}}{p^{knq'}}d\mathbf{y}\bigg)^{1/q'}\\
=&(1-p^{-n})^{1/q'}\sum_{k\in\mathbb{Z}}|\psi(p^{l-k})|p^{kn\lambda}\\
=&(1-p^{-n})^{1/q'}p^{ln\lambda}\sum_{k\in\mathbb{Z}}|%
\psi(p^{l-k})|p^{(l-k)(-n\lambda)-1+1}\\
\leq&C(1-p^{-n})^{1/q'}p^{ln\lambda}\int_{0}^{\infty}\psi(t)t^{-n%
\lambda-1}dt\\
=&C(1-p^{-n})^{1/q'}|\mathbf{x}|_{p}^{n\lambda}\int_{0}^{\infty}\psi(t)t^{-n%
\lambda-1}dt. \end{aligned}
\end{eqnarray*}
We majorized at the penultimate step and last step is courtesy of $|\mathbf{x%
}|_{p}=p^l.$\newline
Letting $K_{1}=C\int_{0}^{\infty}\psi(t)t^{-n\lambda-1}dt$ and by putting
the above values in (\ref{e1}), we get: 
\begin{eqnarray*}
\begin{aligned}|H_{\Phi}|\leq&K_{1}(1-p^{-n})^{1/q'}|\mathbf{x}|_{p}^{n%
\lambda}\|f\|_{\dot{B}^{q,\lambda}(\mathbb{Q}_p^n)}. \end{aligned}
\end{eqnarray*}
Let $A=K_{1}(1-p^{-n})^{1/q^{\prime }}\|f\|_{\dot{B}^{q,\lambda}(\mathbb{Q}%
_p^n)}.$\newline
Since $\lambda<0,$ we have: 
\begin{eqnarray*}
\begin{aligned}\|H_{\Phi}f\|_{W\dot{B}^{q,\lambda}(\mathbb{Q}_p^n)}\leq&%
\sup_{\gamma\in\mathbb{Z}}\sup_{t>0}t|B_{\gamma}|_{H}^{-\lambda-1/q}|\{%
\mathbf{x}\in B_{\gamma}:A|\mathbf{x}|_{p}^{n\lambda}>t\}|^{1/q}\\
=&\sup_{\gamma\in\mathbb{Z}}\sup_{t>0}t|B_{\gamma}|_{H}^{-\lambda-1/q}|\{|%
\mathbf{x}|_{p}\leq p^{\gamma}:|\mathbf{x}|_{p}<(t/A)^{1/n\lambda}\}|^{1/q}.
\end{aligned}
\end{eqnarray*}
If $\gamma\leq\log_{p}(t/A)^{1/n\lambda},$ then for $\lambda<0,$ we get: 
\begin{eqnarray*}
\begin{aligned}&\sup_{t>0}\sup_{\gamma\leq\log_{p}(t/A)^{1/n\lambda}}t|B_{%
\gamma}|_{H}^{-\lambda-1/q}|\{|\mathbf{x}|_{p}\leq
p^{\gamma}:|\mathbf{x}|_{p}<(t/A)^{1/n\lambda}\}|^{1/q}\\
=&\sup_{t>0}\sup_{\gamma\leq\log_{p}(t/A)^{1/n\lambda}}t|B_{\gamma}|_{H}^{-%
\lambda}\\
=&\sup_{t>0}\sup_{\gamma\leq\log_{p}(t/A)^{1/n\lambda}}tp^{-n\gamma\lambda}%
\\ =&A\\
=&K_{1}(1-p^{-n})^{1/q'}|\|f\|_{\dot{B}^{q,\lambda}(\mathbb{Q}_p^n)}.
\end{aligned}
\end{eqnarray*}
Now, if $\gamma>\log_{p}(t/A)^{1/n\lambda},$ then for $\lambda\geq-1/q,$ we
have: 
\begin{eqnarray*}
\begin{aligned}&\sup_{t>0}\sup_{\gamma>\log_{p}(t/A)^{1/n\lambda}}t|B_{%
\gamma}|_{H}^{-\lambda-1/q}|\{|\mathbf{x}|_{p}\leq
p^{\gamma}:|\mathbf{x}|_{p}<(t/A)^{1/n\lambda}\}|^{1/q}\\
=&\sup_{t>0}\sup_{\gamma>\log_{p}(t/A)^{1/n\lambda}}t p^{\gamma
n(-\lambda-1/q)}||\mathbf{x}|_{p}<(t/A)^{1/n\lambda}|^{1/q}\\
=&\sup_{t>0}\sup_{\gamma>\log_{p}(t/A)^{1/n\lambda}}t p^{\gamma
n(-\lambda-1/q)}(t/A)^{1/\lambda q}\\ =&A\\
=&K_{1}(1-p^{-n})^{1/q'}\|f\|_{\dot{B}^{q,\lambda}(\mathbb{Q}_p^n)}.
\end{aligned}
\end{eqnarray*}
Therefore,%
\begin{eqnarray*}
\begin{aligned}\|H_{\Phi
}f\|_{W\dot{B}^{q,\lambda}(\mathbb{Q}_p^n)}\leq&K_{1}(1-p^{-n})^{1/q'}|\|f%
\|_{\dot{B}^{q,\lambda}(\mathbb{Q}_p^n)}. \end{aligned}
\end{eqnarray*}
\end{proof}

\section{Weak and Strong Boundedness of Fractional Hausdorff Operator}

In the current section, we obtain the weak bounds of fractional Hausdorff
operator on weighted $p$-adic weak Lebesgue space. Furthermore, strong type
estimates of same operator for weighted $p$-adic Lorentz space are also
attained.

\begin{theorem}
\label{T1}Let $0\leq\beta<n$ and $1\leq q, r<\infty.$ Let also $%
\min\{\alpha,\gamma\}>-n,$ $w(\mathbf{x})=|\mathbf{x}|_{p}^{\alpha},
\alpha>-n.$ If $\Phi$ is radial function, $\frac{n+\alpha}{q}-\beta=\frac{%
n+\gamma}{r}$ and%
\begin{equation}  \label{yb}
\mathcal{A}^{q^{\prime }}(\psi,q)=\int_{0}^{\infty}|\psi(t)|^{q^{\prime
}}t^{(n+\alpha)(q^{\prime }-1)-\beta q^{\prime }-1}dt,
\end{equation}
then $\|H_{\Phi,\beta}(f)\|_{L^{r,\infty}(|\mathbf{x}|_p^\gamma,\mathbb{Q}%
_p^n)}\leq K_{2}\|f\|_{L^{q}(|\mathbf{x}|_p^\alpha,\mathbb{Q}_p^n)},$\newline
where $K_{2}=C\bigg(\frac{1-p^{-n}}{1-p^{-(n+\gamma)}}\bigg)%
^{1/r}(1-p^{-n})^{1/q^\prime}\mathcal{A}(\psi,q).$
\end{theorem}

\begin{proof}
We first consider: 
\begin{eqnarray*}
\begin{aligned}H_{\Phi,\beta}f(\mathbf{x})=&\int_{\mathbb{Q}_p^n}\frac{\Phi(%
\mathbf{x}|\mathbf{y}|_p)}{|\mathbf{y}|_{p}^{n-\beta}w(\mathbf{y})^{1/q}}f(%
\mathbf{y})w(\mathbf{y})^{1/q}d\mathbf{y}\\=&\int_{\mathbb{Q}_p^n}\frac{%
\Phi(\mathbf{x}|\mathbf{y}|_p)}{|\mathbf{y}|_{p}^{n-\beta+\alpha/q}}f(%
\mathbf{y})w(\mathbf{y})^{1/q}d\mathbf{y} \end{aligned}
\end{eqnarray*}
Applying H\"{o}lder's inequality at the outset to have: 
\begin{eqnarray}
\begin{aligned}[b]\label{E1}|H_{\Phi,\beta}f(\mathbf{x})|\leq&\bigg\{\int_{%
\mathbb{Q}_p^n}\bigg|\frac{\Phi(\mathbf{x}|\mathbf{y}|_p)}{|%
\mathbf{y}|_{p}^{n-\beta+\alpha/q}}\bigg|^{q^\prime}d\mathbf{y}\bigg\}^{1/q^%
\prime}\bigg\{\int_{\mathbb{Q}_p^n}|f(\mathbf{y})|^{q}|\mathbf{y}|_{p}^{%
\alpha}d\mathbf{y}\bigg\}^{1/q}\\
=&\bigg\{\int_{\mathbb{Q}_p^n}\bigg|\frac{\Phi(\mathbf{x}|\mathbf{y}|_p)}{|%
\mathbf{y}|_{p}^{n-\beta+\alpha/q}}\bigg|^{q^\prime}d\mathbf{y}\bigg\}^{1/q^%
\prime}\|f\|_{L^{q}(|\mathbf{x}|_{p}^{\alpha}; \mathbb{Q}_p^n)}.
\end{aligned}
\end{eqnarray}
Sine $\Phi$ is radial function then: 
\begin{eqnarray*}
\begin{aligned}\int_{\mathbb{Q}_p^n}\frac{|\Phi(\mathbf{x}|%
\mathbf{y}|_p)|^{q^\prime}}{|\mathbf{y}|_{p}^{(n-\beta+\alpha/q)q^\prime}}d%
\mathbf{y}=&\sum_{k\in\mathbb{Z}}\int_{S_{k}}\frac{|\psi(p^{l-k})|^{q^%
\prime}}{p^{k(n-\beta+\alpha/q)q^\prime}}d\mathbf{y}\\
=&(1-p^{-n})p^{l((n+\alpha)(1-q^\prime)+\beta
q')}\sum_{k\in\mathbb{Z}}|\psi(p^{l-k})|^{q^\prime}p^{(l-k)((n+\alpha)(q^%
\prime-1)-\beta q')-1+1}\\
\leq&C(1-p^{-n})|\mathbf{x}|_{p}^{-((n+\alpha)(q^\prime-1)-\beta
q')}\int_{0}^{\infty}|\psi(t)|^{q^\prime}t^{(n+\alpha)(q^\prime-1)-\beta
q'-1}dt\\ =&C(1-p^{-n})|\mathbf{x}|_{p}^{-((n+\alpha)(q^\prime-1)-\beta
q')}\mathcal{A}^{q^\prime}(\psi,q). \end{aligned}
\end{eqnarray*}
Therefore, by the stipulation $\frac{n+\alpha}{q}-\beta=\frac{n+\gamma}{r},$
(\ref{E1}) becomes: 
\begin{eqnarray*}
\begin{aligned}|H_{\Phi,\beta}f(\mathbf{x})|\leq &C
(1-p^{-n})^{1/q^\prime}\mathcal{A}(\psi,q)|\mathbf{x}|_{p}^{-(n+\alpha)/q
+\beta}\|f\|_{L^q(|\mathbf{x}|_p^\alpha,\mathbb{Q}_p^n)}%
\\=&C(1-p^{-n})^{1/q^\prime}|\mathbf{x}|_{p}^{-(n+\gamma)/r}\|f\|_{L^q(|%
\mathbf{x}|_{p}^{\alpha};\mathbb{Q}_p^n)}. \end{aligned}
\end{eqnarray*}
Since $C_{1}=C(1-p^{-n})^{1/q^\prime}\mathcal{A}(\psi,q)\|f\|_{L^q(|\mathbf{x%
}|_p^\alpha,\mathbb{Q}_p^n)},$ then for $\lambda>0,$ 
\begin{eqnarray*}
\begin{aligned}\{\mathbf{x}\in\mathbb{Q}_p^n:|H_{\Phi,\beta}f(\mathbf{x})|>%
\lambda\}\subset\{\mathbf{x}\in\mathbb{Q}_p^n:|\mathbf{x}|_{p}\leq(C_{1}/%
\lambda)^{r/n+\gamma}\}. \end{aligned}
\end{eqnarray*}
Therefore, 
\begin{eqnarray*}
\begin{aligned}[b]\label{bb}\|H_{\Phi,\beta}f(\mathbf{x})\|_{L^{r,%
\infty}(|x|_{p}^{\gamma},\mathbb{Q}_p^n)}\leq&\sup_{\lambda>0}\lambda\bigg(%
\int_{\mathbb{Q}_p^n}\chi_{\bigg\{\mathbf{x}\in\mathbb{Q}_p^n:|%
\mathbf{x}|_{p}<(C_{1}/\lambda)^{r/n+\gamma}\bigg\}}(\mathbf{x})|%
\mathbf{x}|_{p}^{\gamma}d\mathbf{x}\bigg)^{1/r}\\
=&\sup_{\lambda>0}\lambda\bigg(\int_{|\mathbf{x}|_{p}<(C_{1}/\lambda)^{r/n+%
\gamma}}|\mathbf{x}|_{p}^{\gamma}d\mathbf{x}\bigg)^{1/r}\\
=&\sup_{\lambda>0}\lambda\bigg(\sum_{j=-\infty}^{\log_{p}(C_{1}/%
\lambda)^{r/n+\gamma}}\int_{S_{j}}p^{j\gamma}d\mathbf{x}\bigg)^{1/r}\\
=&\bigg(\frac{1-p^{-n}}{1-p^{-(n+\gamma)}}\bigg)^{1/r}C_{1}\\
=&C\bigg(\frac{1-p^{-n}}{1-p^{-(n+\gamma)}}\bigg)^{1/r}(1-p^{-n})^{1/q^%
\prime}\mathcal{A}(\psi,q)\|f\|_{L^q(|\mathbf{x}|_p^\alpha,\mathbb{Q}_p^n)}%
\\ =&K_{2}\|f\|_{L^q(|\mathbf{x}|_p^\alpha,\mathbb{Q}_p^n)}. \end{aligned}
\end{eqnarray*}
Hence, $H_{\Phi,\beta}$ has weak type $(q,r).$
\end{proof}

In the next theorem, we will prove strong estimates are also valid for $%
H_{\Phi,\beta}.$

\begin{theorem}
Let $0\leq\beta<n$, let $1\leq q, r<\infty.$ Let also $\min\{\alpha,\gamma%
\}>-n,$ $w(\mathbf{x})=|\mathbf{x}|_{p}^{\alpha}, \alpha>-n.$ If $\Phi$ is
radial function, let $\frac{n+\alpha}{q}-\beta=\frac{n+\gamma}{r}$ and
equation (\ref{yb}) is valid for $q\pm\epsilon$ instead of $q,$ then 
\begin{eqnarray*}
\begin{aligned}\|H_{\Phi,\beta}f\|_{L^{r,s}(|\mathbf{x}|_{p}^{\alpha},%
\mathbb{Q}_p^n)}\preceq\|f\|_{L^{q,s}(|\mathbf{x}|_{p}^{\alpha},%
\mathbb{Q}_p^n)}. \end{aligned}
\end{eqnarray*}
\end{theorem}

\begin{proof}
Since $1<q,r<\infty,$ so we can find $\epsilon$ such that $1<q-\epsilon$ and 
$1<r-\epsilon.$ Then, by Theorem (\ref{T1}) $H_{\Phi,\beta}$ has weak types $%
(q-\epsilon,r-\epsilon)$ and $(q+\epsilon,r+\epsilon).$ Desired result is
acquired by using Theorem \ref{T^}.
\end{proof}

\section{Weak and Strong Boundedness of commutator of Fractional Hausdorff
Operator}

This section comprises of weak boundedness of commutator of $p$-adic
fractional Hausdorff operator on weighted $p$-adic weak Lebesgue space. At
the section end, we also obtain the strong type estimates for the same
operator on power weighted $p$-adic Lorentz space.

\begin{theorem}
\label{T2}Let $1<q<r<\infty,$ $0<\delta<1.$ Let also $\min\{\alpha,\gamma%
\}>-n,$ $(\beta+\delta)-\frac{n+\alpha}{q}=-\frac{n+\gamma}{r},$ $w(\mathbf{x%
})=|\mathbf{x}|_{p}^{\alpha}, \alpha>-n.$ If $\Phi$ is radial function, $%
b\in\Lambda_{\delta}(\mathbb{Q}_p^n)$ and 
\begin{eqnarray}
\begin{aligned}[b]\label{E3}C\int_{0}^{\infty}\psi^{q^\prime}(t)t^{(q^%
\prime-1)(n+\alpha)-\beta q^\prime-1}\max(1, t^{-\delta
q^\prime})dt=K_{3,q},\end{aligned}
\end{eqnarray}
then 
\begin{eqnarray}
\begin{aligned}[b]\label{nn1}\|H_{\Phi,\beta}^{b}f\|_{L^{r,
\infty}(|\mathbf{x}|_{p}^{\alpha}, \mathbb{Q}_p^n)}\leq
K_{4}\|f\|_{L^{q}(|\mathbf{x}|_{p}^{\alpha}, \mathbb{Q}_p^n)}, \end{aligned}
\end{eqnarray}
where $K_{4}=K_{3,q}\bigg (\frac{1-p^{-n}}{1-p^{-(n+\gamma)}}\bigg)%
^{1/r}(1-p^{-n})\|b\|_{\Lambda_{\delta}(\mathbb{Q}_p^n)}.$
\end{theorem}

\begin{proof}
We first take: 
\begin{eqnarray*}
\begin{aligned}H_{\Phi,\beta}f(\mathbf{x})=\int_{\mathbb{Q}_p^n}\frac{\Phi(%
\mathbf{x}|\mathbf{y}|_p)}{|\mathbf{y}|_{p}^{n-\beta+\alpha/q}}f(%
\mathbf{y})|\mathbf{y}|_{p}^{\alpha/q}d\mathbf{y} \end{aligned}
\end{eqnarray*}
By definition of Lipschitz space, we have: 
\begin{eqnarray*}
\begin{aligned}|H_{\Phi,\beta}^{b}(f)(\mathbf{x})|\leq&\bigg|\int_{%
\mathbb{Q}_p^n}\frac{\Phi(\mathbf{x}|\mathbf{y}|_{p})}{|\mathbf{y}|_{p}^{n-%
\beta+\alpha/q}}(b(\mathbf{x})-b(\mathbf{0}))f(\mathbf{y})|\mathbf{y}|_{p}^{%
\alpha/q}d\mathbf{y}\bigg|\\\quad&+\bigg|\int_{\mathbb{Q}_p^n}\frac{\Phi(%
\mathbf{x}|\mathbf{y}|_{p})}{|\mathbf{y}|_{p}^{n-\beta+\alpha/q}}(b(%
\mathbf{y})-b(\mathbf{0}))f(\mathbf{y})|\mathbf{y}|_{p}^{\alpha/q}d%
\mathbf{y}\bigg|\\
\leq&\|b\|_{\Lambda_{\delta}(\mathbb{Q}_p^n)}|\mathbf{x}|_{p}^{\delta}\int_{%
\mathbb{Q}_p^n}\frac{\Phi(\mathbf{x}|\mathbf{y}|_{p})}{|\mathbf{y}|_{p}^{n-%
\beta+\alpha/q}}f(\mathbf{y})|\mathbf{y}|_{p}^{\alpha/q}d\mathbf{y}\\\quad&+%
\|b\|_{\Lambda_{\delta}(\mathbb{Q}_p^n)}\int_{\mathbb{Q}_p^n}\frac{\Phi(%
\mathbf{x}|\mathbf{y}|_{p})}{|\mathbf{y}|_{p}^{n-\beta+\alpha/q-\delta}}f(%
\mathbf{y})|\mathbf{y}|_{p}^{\alpha/q}d\mathbf{y}\\ =&I_{1}+I_{2}.
\end{aligned}
\end{eqnarray*}
We evaluate $I_{2}$ first. By H\"{o}lder's inequality, we get: 
\begin{eqnarray}
\begin{aligned}[b]\label{E2}I_{2}\leq&\|b\|_{\Lambda_{\delta}(%
\mathbb{Q}_p^n)}\bigg\{\int_{\mathbb{Q}_p^n}\bigg|\frac{\Phi(\mathbf{x}|%
\mathbf{y}|_p)}{|\mathbf{y}|_{p}^{n-\beta+\alpha/q-\delta}}\bigg|^{q^%
\prime}d\mathbf{y}\bigg\}^{1/q^\prime}\bigg\{\int_{\mathbb{Q}_p^n}|f(%
\mathbf{y})|^{q}|\mathbf{y}|_{p}^{\alpha}d\mathbf{y}\bigg\}^{1/q}\\
=&\|b\|_{\Lambda_{\delta}(\mathbb{Q}_p^n)}\bigg\{\int_{\mathbb{Q}_p^n}\bigg|%
\frac{\Phi(\mathbf{x}|\mathbf{y}|_p)}{|\mathbf{y}|_{p}^{n-\beta+\alpha/q-%
\delta}}\bigg|^{q^\prime}d\mathbf{y}\bigg\}^{1/q^\prime}\|f\|_{L^{q}(|%
\mathbf{x}|_{p}^{\alpha},\mathbb{Q}_p^n)}. \end{aligned}
\end{eqnarray}
If $|\mathbf{x}|_{p}=p^{l}, l\in\mathbb{Z},$ then repeating the same process
as in Theorem (\ref{T1}), we arrive at: 
\begin{eqnarray*}
\begin{aligned}\int_{\mathbb{Q}_p^n}\frac{|\Phi(\mathbf{x}|%
\mathbf{y}|_p)|^{q^\prime}}{|\mathbf{y}|_{p}^{(n-\beta+\alpha/q-\delta)q^%
\prime}}d\mathbf{y}=&\sum_{k\in\mathbb{Z}}\int_{S_{k}}\frac{|%
\psi(p^{l-k})|^{q^\prime}}{p^{k(n-\beta+\alpha/q-\delta)q^\prime}}d%
\mathbf{y}\\ =&(1-p^{-n})
\sum_{k\in\mathbb{Z}}\psi^{q^\prime}(p^{l-k})p^{(l-k)((q^\prime-1)(n+%
\alpha)-(\beta+\delta)q^\prime-1+1)}\\\quad&\times|\mathbf{x}|_{p}^{-((n+%
\alpha)(q'-1)-(\beta+\delta)q')}\\
\leq&C(1-p^{-n})\int_{0}^{\infty}\psi^{q^\prime}(t)t^{(q^\prime-1)(n+%
\alpha)-(\beta+\delta)q^\prime-1}dt|\mathbf{x}|_{p}^{-((n+\alpha)(q'-1)-(%
\beta+\delta)q')}. \end{aligned}
\end{eqnarray*}
Making use of above value, (\ref{E2}) becomes: 
\begin{eqnarray*}
\begin{aligned}I_{2}\leq C
(1-p^{-n})\|b\|_{\Lambda_{\delta}(\mathbb{Q}_p^n)}|\mathbf{x}|_{p}^{-(n+%
\alpha)/q+(\beta+\delta)}\bigg(\int_{0}^{\infty}\psi^{q^\prime}(t)t^{(q^%
\prime-1)(n+\alpha)-(\beta+\delta)q^\prime-1}dt\bigg)^{1/q^\prime}\|f%
\|_{L^{q}(|\mathbf{x}|_{p}^{\alpha},\mathbb{Q}_p^n)}. \end{aligned}
\end{eqnarray*}
By inserting $|\mathbf{y}|_{p}^{n-\beta+\alpha/q}$ instead of $|\mathbf{y}%
|_{p}^{n-\beta+\alpha/q-\delta}$, we get: 
\begin{eqnarray*}
\begin{aligned}I_{1}\leq C
(1-p^{-n})\|b\|_{\Lambda_{\delta}(\mathbb{Q}_p^n)}|\mathbf{x}|_{p}^{-(n+%
\alpha)/q+(\beta+\delta)}\bigg(\int_{0}^{\infty}\psi^{q^\prime}(t)t^{(q^%
\prime-1)(n+\alpha)-\beta
q^\prime-1}dt\bigg)^{1/q^\prime}\|f\|_{L^{q}(|\mathbf{x}|_{p}^{\alpha},%
\mathbb{Q}_p^n)}. \end{aligned}
\end{eqnarray*}
Since $(\beta+\delta)-\frac{n+\alpha}{q}=-\frac{n+\gamma}{r},$ and by (\ref%
{E3}), we obtain: 
\begin{eqnarray*}
\begin{aligned}|H_{\Phi,\beta}^{b}(f)(\mathbf{x})|\leq K_{3,q}
(1-p^{-n})\|b\|_{\Lambda_{\delta}(\mathbb{Q}_p^n)}|\mathbf{x}|_{p}^{-(n+%
\gamma)/r}\|f\|_{L^{q}(|\mathbf{x}|_{p}^{\alpha},\mathbb{Q}_p^n)}
\end{aligned}
\end{eqnarray*}
Let $C_{3}=K_{3,q}(1-p^{-n})\|b\|_{\Lambda_{\delta}(\mathbb{Q}%
_p^n)}\|f\|_{L^{q}(|\mathbf{x}|_{p}^{\alpha},\mathbb{Q}_p^n)},$ then for all 
$\lambda>0,$ we have: 
\begin{eqnarray*}
\begin{aligned}\{\mathbf{x}\in\mathbb{Q}_p^n:|H_{\Phi,\beta}^{b}f(%
\mathbf{x})|>\lambda\}\subset\{\mathbf{x}\in\mathbb{Q}_p^n:|\mathbf{x}|_{p}%
\leq(C_{3}/\lambda)^{r/n+\gamma}\}. \end{aligned}
\end{eqnarray*}
Ultimately, 
\begin{eqnarray*}
\begin{aligned}\|H_{\Phi,\beta}^{b}f(\mathbf{x})\|_{L^{r,\infty}(|%
\mathbf{x}|_{p}^{\gamma},\mathbb{Q}_p^n)}\leq&\sup_{\lambda>0}\lambda\bigg(%
\int_{\mathbb{Q}_p^n}\chi_{\bigg\{\mathbf{x}\in\mathbb{Q}_p^n:|%
\mathbf{x}|_{p}<(C_{3}/\lambda)^{r/n+\gamma}\bigg\}}(\mathbf{x})|%
\mathbf{x}|_{p}^{\gamma}d\mathbf{x}\bigg)^{1/r}\\
=&\sup_{\lambda>0}\lambda\bigg(\int_{|\mathbf{x}|_{p}<(C_{3}/\lambda)^{r/n+%
\gamma}}|\mathbf{x}|_{p}^{\gamma}d\mathbf{x}\bigg)^{1/r}\\=&\bigg(%
\frac{1-p^{-n}}{1-p^{-(n+\gamma)}}\bigg)^{1/r}C_{3}\\ =&K_{3,q}\bigg
(\frac{1-p^{-n}}{1-p^{-(n+\gamma)}}\bigg)^{1/r}(1-p^{-n})\|b\|_{\Lambda_{%
\delta}}\|f\|_{L^{q}(|\mathbf{x}|_{p}^{\alpha},\mathbb{Q}_p^n)}\\
=&K_{4}\|f\|_{L^{q}(|\mathbf{x}|_{p}^{\alpha},\mathbb{Q}_p^n)}. \end{aligned}
\end{eqnarray*}
\end{proof}

Next, we will show that the strong type estimates also hold for $%
H_{\Phi,\beta}^{b}.$

\begin{theorem}
Let $1<q<r<\infty,$ $0<\delta<1.$ Let also $\min\{\alpha,\gamma\}>-n,$ $%
(\beta+\delta)-\frac{n+\alpha}{q}=-\frac{n+\gamma}{r},$ $w(\mathbf{x})=|%
\mathbf{x}|_{p}^{\alpha}, \alpha>-n.$ If $\Phi$ is radial function, $%
b\in\Lambda_{\delta}(\mathbb{Q}_p^n)$ and equation (\ref{E3}) is true for $%
q\pm\epsilon$ instead of $q$ where $0\leq\epsilon<\epsilon_{0},$ then 
\begin{eqnarray*}
\begin{aligned}\|H^{b}_{\Phi,\beta}f\|_{L^{r,s}(|\mathbf{x}|_{p}^{\alpha},%
\mathbb{Q}_p^n)}\preceq\|f\|_{L^{q,s}(|\mathbf{x}|_{p}^{\alpha},%
\mathbb{Q}_p^n)}. \end{aligned}
\end{eqnarray*}
\end{theorem}

\begin{proof}
Since $q, r \in(1, \infty),$ so $q\in(1,n/\delta),$ we can find $%
0\leq\epsilon<\epsilon_{0}$ such that $q_{1}=q-\epsilon\in(1,n/\delta)$ and $%
q_{2}=q+\epsilon\in(1,n/\delta).$ Also, we can choose $r_{1}$ and $r_{2}$
such that $r_{1}<r<r_{2}$ which satisfies 
\begin{equation*}
(\beta+\delta)-\frac{n+\alpha}{q_{i}}=-\frac{n+\gamma}{r_{i}}, i=1,2.
\end{equation*}
Using Theorem \ref{T2}, we have: 
\begin{eqnarray*}
\begin{aligned}[b]\label{nn}\|H_{\Phi,\beta}^{b}f\|_{L^{r_{i},
\infty}(|\mathbf{x}|_{p}^{\gamma},
\mathbb{Q}_p^n)}\preceq\|f\|_{L^{q_{i}}(|\mathbf{x}|_{p}^{\alpha},
\mathbb{Q}_p^n)}. \end{aligned}
\end{eqnarray*}
But the equality $1/q=\vartheta/q_{1}+(1-\vartheta)/q_{2}$ implies a similar
equality $1/r=\vartheta/r_{1}+(1-\vartheta)/r_{2}.$ Required result is
obtained by using Theorem \ref{T^}.
\end{proof}


\begin{thebibliography}{99}
\bibitem{A} Andersen KF. Boundedness of Hausdorff Operator on $L^{p}(\mathbb{%
R}^{n}),$ $H^{1}(\mathbb{R}^{n})$ and $BMO (\mathbb{R}^{n}).$ Acta Sci Math
(Szeged), 2003; \textbf{69}:  409-418.

\bibitem{ADFV} Arefeva IY, Dragovich B, Frampton P, Volovich IV. The wave
function of the universe and $p$-Adic gravity. Mod. Phys. Lett A, 1991; 
\textbf{6}: 4341-4358.

\bibitem{ABKO1} Avestisov AV, Bikulov AH, Kozyrev SV, Osipov VA. Application
of $p$-Adic analysis to models of spontaneous breaking of replica symmetry.
J. Phys. A: Math. Gen. 1999; \textbf{32}: 8785-8791.

\bibitem{ABKO} Avestisov AV, Bikulov AH, Kozyrev SV, Osipov VA. $p$-adic
models of ultrametric diffusion constrained by hierarchical energy
landscape. J. Phys. A: Math. Gen. 2002; \textbf{35}: 177-189.

\bibitem{ABO} Avestisov AV, Bikulov AH, Osipov VA. $p$-adic description of
Characterization relaxation in complex system. J. Phys. A: Math. Gen. 2003; 
\textbf{36}: 4239-4246.

\bibitem{DGSK} Dubischar D, Gundlach VM, Steinkamp O, Khrennikov A. A $p$%
-Adic model for the process of thinking disturbed by physiological and
information noise, J. Theor. Biol. 1999; \textbf{197(4)}: 451-467.

\bibitem{FZ} Fan D, Zhao F. Sharp constants for multivariate Hausdorff $q$%
-inequalities. J. Aust. Math. Soc. 2018; doi:10.1017/S1446788718000113.

\bibitem{GHZ} Gao F, Hu X, Zhong C. Sharp weak estimates for Hardy-type
Operators, Ann. Funct. Anal. 2016; \textbf{7(3)}: 421--433.

\bibitem{GZ} Gao G, Zhao F. Sharp weak bounds for Hausdorff operators. Anal.
Math. 2015; \textbf{41}: 163--173.

\bibitem{HJ} Haixia Y, Junferg L. Sharp weak estimates for $n$-dimensional
fractional Hardy Operators. Front. Math. China. 2018; \textbf{13(2)}:
449--457.

\bibitem{SH} Haran S. Analytic potential theory over the $p$- adics. Ann.
Inst. Fourier (Grenoble). 1993; \textbf{43}: 905-944.

\bibitem{SR} Haran S. Riesz potential and explicit sums in arithemtic.
Invent. Math. 1990; \textbf{101}: 697-703.

\bibitem{K} Kochubei AN. Stochastic integrals and stochastic differential
equations over the field of $p$-Adic numbers. Potential Analysis. 1997; 
\textbf{6}: 105-125.

\bibitem{SVK} Kozyrev SV. Methods and applications of ultrametric and $p$%
-adic analysis: From wavelet theory to biophysics, Proc. Steklov. Inst.
Math. 2011; \textbf{274}: 1-84.

\bibitem{LL} Lerner AK, Liflyand E. Multidimensional Hausdorff operators on
the real Hardy Space. Acta. Sci. Math(Szeged). 2007; \textbf{83}: 79-86.

\bibitem{LS} Lin X, Sun L. Some estimates on the Hausdorff operator. Acta
Sci. Math. (Szeged). 2012; \textbf{78}: 669-681.

\bibitem{PS} Parisi G, Sourlas N. $p$-Adic numbers and replica symmetry
break. Eur. J. Phy. B. 2000; \textbf{14}: 535-542.

\bibitem{VV} Vladimirov VS, Volovich IV. $p$-Adic quantum mechanics. Commun.
Math. Phys. 1989; \textbf{123}: 659-676.

\bibitem{VVZ} Vladimirov IV, Volovich IV, Zelenov EI. $p$-Adic Analysis and
Mathematical Physics. World Scientific, Singapore. 1994.

\bibitem{WF1} Wu QY, Fu ZW. Hardy-Littlewood-Sobolev Inequalities on $p$%
-adic Central Morrey Spaces. Journal of Function Spaces Volume 2015; Article
ID 419532: 7 pages; http://dx.doi.org/10.1155/2015/419532.

\bibitem{WF} Wu QY, Fu ZW. Weighted $p$-Adic Hardy operators and their
commutators on $p$-Adic central Morrey spaces. Malays. Math. Sci. Soc. 2015; 
\textbf{40}: 635-654.
\end{thebibliography}
\end{document}